\documentclass{amsart}
\usepackage{physics}
\usepackage{hyperref}

\newtheorem{theorem}{Theorem}[section]
\newtheorem{lemma}[theorem]{Lemma}
\newtheorem{proposition}[theorem]{Proposition}

\theoremstyle{definition}

\newtheorem{example}[theorem]{Example}

\theoremstyle{remark}
\newtheorem{remark}[theorem]{Remark}

\numberwithin{equation}{section}

\begin{document}

\title[Thickness via partial differential equations]{On a calculation method of the thickness via partial differential equations}

\author{Atsushi~Nakayasu}
\address{Graduate School of Engineering, The University of Tokyo, Yayoi 2-11-16, Bunkyo-ku, Tokyo 113-8656, Japan}
\email{ankys@g.ecc.u-tokyo.ac.jp}

\author{Takayuki~Yamada}
\address{Graduate School of Engineering, The University of Tokyo, Yayoi 2-11-16, Bunkyo-ku, Tokyo 113-8656, Japan}
\email{t.yamada@mech.t.u-tokyo.ac.jp}


\date{\today}



\begin{abstract}
This paper presents a mathematical analysis of an elliptic partial differential equation (PDE) designed to compute the geometric thickness of a given shape.
The PDE-based formulation provides a direct and systematic approach to evaluate thickness through the elliptic equation, whose solution yields a vector field from which the thickness is extracted as the divergence.
While the convergence of this PDE-based thickness to the geometric thickness had been rigorously justified only for simple geometries such as intervals and straight bands, its validity for more general shapes remained open.
In this work, we extend the analysis to annular domains, where curvature effects are nontrivial.
We prove that the PDE-based thickness converges to the geometric thickness as the diffusion parameter tends to zero
by estimating the difference between two notions of thickness with the square root of the diffusion parameter.
Explicit expressions involving modified Bessel functions are obtained for annuli, together with sharp inequalities for their ratios.
These results provide a rigorous mathematical foundation for the PDE-based thickness
and demonstrate its potential as a reliable tool in shape analysis and topology optimization.
\end{abstract}

\maketitle


\section{Introduction}

A novel partial differential equation (PDE)-based framework has recently been introduced for extracting geometric features of shapes, including thickness, orientation, and skeleton \cite{r_y19}.
In this formulation, a vector field is obtained as a solution of a linear elliptic PDE posed on an extended domain,
and the thickness is extracted as the divergence of this field.
A distinguishing feature of this approach is that it is expected to compute the thickness itself.
This provides a direct PDE characterization of a geometric notion of thickness that has traditionally been defined only in purely geometric terms.
The formulation is therefore attractive from both theoretical and computational viewpoints, as it enables the use of analytical tools from elliptic PDE theory as well as standard finite element implementations.

Nevertheless, the PDE-based thickness is not a priori identical to the intuitive geometric thickness.
The solution of the elliptic equation depends not only on the shape domain but also on the surrounding fictitious domain and on a diffusion parameter.
Hence, two fundamental questions arise: whether the PDE-based thickness converges to the geometric thickness in the singular limit of vanishing diffusion, and how the approximation error behaves for finite parameters.
Establishing these properties is indispensable for justifying the PDE-based thickness as a reliable model.

An initial analysis was given in the authors' previous work \cite{r_ny25},
where the elliptic equation was studied for one-dimensional intervals and infinite straight bands of constant thickness.
By constructing explicit solutions and employing the maximum principle and interior $H^1$ estimates, it was shown that the PDE thickness converges to the geometric thickness as the diffusion parameter tends to zero.
These results provided a first rigorous validation of the method, but were limited to relatively simple geometries.

From the viewpoint of applications, a mathematically consistent notion of thickness is highly relevant to topology optimization, where minimum and maximum thickness constraints are indispensable for manufacturability.
Classical methods, such as signed distance function method or projection techniques (see Allaire et al.\ \cite{r_ajm16}, Carroll--Guest \cite{r_cg22}), act only indirectly on thickness.
In contrast, the present method offers the possibility of incorporating thickness itself into optimization models in a direct and PDE-consistent manner \cite{r_soy22}, \cite{r_y19b}.
A rigorous justification of this approach is thus not only of mathematical interest but also of practical importance in design optimization.

The aim of the present paper is to extend the analytical foundation of the elliptic equation beyond the simple cases treated in \cite{r_ny25}.
We investigate prototypical but nontrivial geometries, including bands and annuli, where explicit solutions or sharp estimates can be obtained.
For these settings we prove that the PDE-based thickness converges to the geometric thickness almost everywhere as the diffusion parameter tends to zero.
Moreover, we derive quantitative error estimates that characterize the rate of convergence with the square root of the diffusion parameter and additional exponentially small terms depending on the geometry.
In the annulus case, the analysis involves explicit representations in terms of modified Bessel functions and sharp inequalities for their ratios.
These results provide a broader and deeper understanding of the elliptic equation and PDE-based thickness, reinforcing its role as a mathematically justified and practically applicable tool for computing thickness in shape analysis and topology optimization.

\section{Preliminaries}

Let $\Omega$ be a domain (connected open set) representing the shape to be analyzed in the Euclidean space $\mathbb{R}^N$.
The \emph{(local) thickness} $\bar{T}_\Omega$ is a scalar field on $\Omega$ defined as the diameter of the largest inscribed ball at each point, that is,
\[
\bar{T}_\Omega(\vb*{x}) := \sup\{ \text{The diameter of the closed ball $B$} \mid \vb*{x} \in B \subset \Omega \}.
\]

\begin{example}[Interval shape]
For example, the thickness of an interval
\[
\Omega = (f_l, f_r)
\]
in one dimension with $f_l < f_r$ is the constant $\bar{T}_\Omega = f_r-f_l$.
\end{example}

\begin{example}[Band shape]
Similarly, the thickness of a band
\[
\Omega = \mathbb{R}\times(f_l, f_r)
\]
in two dimension with $f_l < f_r$ is the constant $\bar{T}_\Omega = f_r-f_l$.
\end{example}

\begin{example}[Annulus shape]
Even for an annulus
\[
\Omega = \{ (x, y) \in \mathbb{R}^2 \mid f_l^2 < x^2+y^2 < f_r^2 \}
\]
in two dimension with $0 < f_l < f_r$,
the thickness is the constant $\bar{T}_\Omega = f_r-f_l$.
\end{example}

In order to define the PDE-based thickness, prepare another domain $D$ in $\mathbb{R}^N$ to formulate the elliptic equation such that $\overline{\Omega} \subset D$ and the boundary $\partial D$ is $C^1$.
We call $D\setminus\overline{\Omega}$ the \emph{void domain}
while $\Omega$ is called the \emph{shape domain}.
The boundary between the shape domain and the void domain is denoted by $\Gamma := \partial\Omega$.
To distinguish between the shape domain and the void domain, define the characteristic function $\chi_\Omega \in L^\infty(D)$ by
\[
\chi_\Omega (x) =
\begin{cases}
1 & \text{if $x \in \Omega$,} \\
0 & \text{if $x \notin \Omega$.}
\end{cases}
\]

The elliptic equation to be analyzed in this paper is a vector-valued linear partial differential equation of the form
\begin{equation}
\label{e_eq}
\begin{cases}
-a \Delta \vb*{s}+(1-\chi_\Omega) \vb*{s} = -\nabla \chi_\Omega \quad \text{in $D$,} \\
\vb*{s} = \vb*{0} \quad \text{on $\partial D$.}
\end{cases}
\end{equation}
Here, $a$ is a fixed positive constant parameter for regularizing the solution $\vb*{s}$, and the case of $a \to 0$ is of most interest.
The solution $\vb*{s}$ is a state variable for extracting shape features, which is an $N$-dimensional vector field.
For example, it is known that $\vb*{s}$ converges to the normal vector when $a \to 0$ \cite{r_hktmy20}.
Another example is that a simplified equation is to be used for calculating the signed distance functions \cite{r_hmosy24}.

More precisely, we understand the equation in a weak sense.
Let $H^1(D)$ be the set of scalar fields $s(\vb*{x}) = s(x_1, \cdots, x_N)$ on $D$ such that $s$ itself and each partial derivative $s_{x_1}, \dots, s_{x_N}$ are an $L^2$ functions,
and $H^1(D)^N$ be the set of vector fields $\vb*{s}(\vb*{x}) = (s^1(\vb*{x}), \cdots, s^N(\vb*{x}))$ whose each component $s^1, \cdots, s^N$ belongs to $H^1(D)$.
Note that $H^1(D)$ and $H^1(D)^N$ are Hilbert spaces with the inner products defined by
\[
(s, u)_a := a \int_D \nabla s\cdot\nabla u+\int_{D\setminus\Omega} s u
\quad \text{for $s, u \in H^1(D)$,}
\]
\[
(\vb*{s}, \vb*{u})_a := a \int_D \nabla\vb*{s}\colon\nabla\vb*{u}+\int_{D\setminus\Omega} \vb*{s}\cdot\vb*{u}
\quad \text{for $\vb*{s}, \vb*{u} \in H^1(D)^N$,}
\]
respectively.
Here, $\nabla\vb*{s}\colon\nabla\vb*{u}$ is the Frobenius inner product of two matrices $\nabla\vb*{s}$ and $\nabla\vb*{u}$, i.e., $\nabla\vb*{s}\colon\nabla\vb*{u} = \sum_{i, j=1}^N s_{x_j}^i u_{x_j}^i$.
Let $H_0^1(D)$ be the subspace of $H^1(D)$ spanned by the infinitely differentiable functions with compact support in $D$,
and $H_0^1(D)^N$ be the set of vector fields whose each component belongs to $H_0^1(D)$.
Note that $H^1_0(\mathbb{R}^N) = H^1(\mathbb{R}^N)$.

Now, the weak form of the equation \eqref{e_eq} is to find $\vb*{s} \in H_0^1(D)^N$ such that
\begin{equation}
\label{e_wf}
a \int_D \nabla\vb*{s}:\nabla\vb*{u}+\int_{D\setminus\Omega} \vb*{s}\cdot\vb*{u} = \int_\Omega \operatorname{div}\vb*{u}
\quad \forall \vb*{u} \in H_0^1(D)^N.
\end{equation}
We remark that there is a unique (weak) solution $\vb*{s} \in H_0^1(D)^N$ to \eqref{e_wf}
provided that $\Omega$ has a finite area.
Indeed, one can apply the Lax-Milgram theorem thanks to the estimate
\[
\begin{aligned}
\abs{\int_\Omega \operatorname{div}\vb*{u}}
&\le \int_\Omega \abs{\operatorname{div}\vb*{u}}
\le \norm{\operatorname{div}\vb*{u}}_{L^2(\Omega)} \norm{1}_{L^2(\Omega)} \\
&\le \sqrt{N} \norm{1}_{L^2(\Omega)} \norm{\nabla\vb*{u}}_{L^2(\Omega)} \\
&\le \sqrt{\frac{N\abs{\Omega}}{a}} \norm{\vb*{u}}_a,
\end{aligned}
\]
where $\norm{\vb*{u}}_a := \sqrt{(\vb*{u}, \vb*{u})_a}$ and $\abs{\Omega}$ is the area of the shape domain $\Omega$.

Also note that the right-hand side can be further calculated,
and letting $\vb*{n}$ is the outward unit normal vector on the boundary of $\Omega$,
we have the identity
\[
\int_\Omega \operatorname{div}\vb*{u} = \int_\Gamma \vb*{u}\cdot\vb*{n}.
\]
We will use this form to quadrature the exact solution of the elliptic equation \eqref{e_wf} in some cases.

Define the \emph{PDE-based thickness} $T_{D, \Omega}^a$ on $\Omega$ as the limit of
\[
T_{D, \Omega}^a = \frac{2}{\sqrt{a} \operatorname{div}\vb*{s}}.
\]
We remark that the PDE-based thickness might contain the singular points where $\operatorname{div}\vb*{s} = 0$.
To avoid such a problem, we mainly consider the inverse of the PDE-based thickness, i.e.,
\[
\frac{1}{T_{D, \Omega}^a} = \frac{\sqrt{a}}{2} \operatorname{div}\vb*{s},
\]
which is guaranteed to be an $L^2$ function on $\Omega$.

The equation \eqref{e_wf} is non-homogeneous because the term on the right-hand side exists.
A \emph{homogeneous equation} has no terms on the right-hand side, which is written as
\begin{equation}
\label{e_hswf}
a \int_D \nabla\vb*{d}:\nabla\vb*{u}+\int_{D\setminus\Omega} \vb*{d}\cdot\vb*{u} = 0
\quad \forall \vb*{u} \in H_0^1(D)^N.
\end{equation}
The difference $\vb*{d} = \vb*{s}_2-\vb*{s}_1 \in H^1(D)^N$ between two solutions $\vb*{s}_1, \vb*{s}_2 \in H^1(D)^N$ of a non-homogeneous equation \eqref{e_wf} (which do not necessarily satisfy the boundary conditions) is a solution to the homogeneous equation \eqref{e_hswf}.

We prepare two lemmas for the analysis of the equation.
The first one is the maximum principle \cite{r_km12}.

\begin{lemma}[Maximum principle for homogeneous equations]
\label{t_mp}
Let $\vb*{d} \in H^1(D)^N$ be a weak solution of the homogeneous equation \eqref{e_hswf} not necessarily satisfying the boundary conditions.
Then, we have
\[
\operatorname*{ess\ sup}_{D} \abs*{\vb*{d}} \le \operatorname*{ess\ sup}_{\partial D} \abs*{\vb*{d}},
\]
where $\operatorname*{ess\ sup}$ is the essential supremum and the right-hand side is understood as the trace sense.
\end{lemma}

\begin{proof}
Let $c = \operatorname*{ess\ sup}_{\partial D} \abs*{\vb*{d}}$.
First consider the scalar-valued equation
\begin{equation}
\label{e_hs1wf}
a\int_D \nabla d\cdot\nabla u+\int_{D\setminus \Omega}d u = 0
\quad \forall u \in H^1_0(D).
\end{equation}
In order to show $\abs{d} \le c$ a.e.\ in $D$, we only show $d \le c$ a.e.\ in $D$ because we can show $d \ge -c$ by a similar way.
Let us take the test function
\[
u(\vb*{x}) = k_+(\vb*{x}) = \max\{ 0, d(\vb*{x})-c \}.
\]
Since $u \ge 0$ a.e.\ on $D$ and $u \in H^1_0(D)$ by the assumption,
it follows from the weak form \eqref{e_hs1wf} that
\[
a\int_D \nabla d\cdot\nabla k_+ +\int_{D\setminus \Omega}d k_+ \le 0.
\]
Noting that if $k_+ > 0$, then $d > c \ge 0$,
we have
\[
a\int_D \abs{\nabla k_+}^2 \le 0.
\]
Therefore, $\nabla k_+ = 0$ on $D$, so $k_+ = 0$, or $d \le c$.
Thus, we have shown that $\abs{d} \le c$ a.e.\ in $D$.

For the vector-valued case, by the above argument we see that
\[
\operatorname*{ess\ sup}_D \abs{\vb*{d}\cdot \vb*{e}}
\le \operatorname*{ess\ sup}_{\partial D} \abs{\vb*{d}\cdot \vb*{e}}
\le c
\]
for all unit vector $\vb*{e} \in \mathbb{R}^N$
since $d^{\vb*{e}}(\vb*{x}) := \vb*{d}(\vb*{x})\cdot\vb*{e}$ is a weak solution of the scalar-valued equation \eqref{e_hs1wf}.
Therefore, we have the desired inequality of this lemma.
\end{proof}

The interior $H^1$ estimate is mentioned in \cite[Proof of Theorem 1 in Subsection 6.3.1]{r_e10} and \cite[Problem 8.2]{r_gt83},
but here we give an abstract version since we make a complete proof including the coefficients in the later sections.

\begin{lemma}[Interior $H^1$ estimate for homogeneous equations]
\label{t_ire_ab}
Let $\vb*{d} \in H^1(D)^N$ be a weak solution of the homogeneous equation \eqref{e_hswf} not necessarily satisfying the boundary conditions.
Then, for a non-negative function $c \in W^{2, \infty}_0 (D)$ such that $c = 1$ on $\Omega$, the inequality
\[
\int_\Omega \abs*{\nabla \vb*{d}}^2 \le \frac{1}{2}\int_D \abs*{\Delta c}\abs*{\vb*{d}}^2
\]
holds.
\end{lemma}

\begin{proof}
By testing $\vb*{u} = c\vb*{d} \in H^1_0(D)$ in \eqref{e_hswf},
since $\nabla \vb*{u} = \vb*{d}\otimes\nabla c+c\nabla \vb*{d}$,
we have
\[
a\int_D c\abs{\nabla \vb*{d}}^2+a\int_D \nabla \vb*{d}:\vb*{d}\otimes\nabla c+\int_{D\setminus \Omega}c\abs{\vb*{d}}^2 = 0.
\]
Here, $\otimes$ represents the tensor product and the identity $\nabla \vb*{d}:\vb*{d}\otimes\nabla c = \frac{1}{2}\nabla(\abs{\vb*{d}}^2)\cdot\nabla c$ gives us
\[
a\int_D c\abs{\nabla \vb*{d}}^2-a\int_D \frac{1}{2}\Delta c\abs{\vb*{d}}^2+\int_{D\setminus \Omega}c\abs{\vb*{d}}^2 = 0.
\]
Therefore,
\[
\int_{\Omega} \abs{\nabla \vb*{d}}^2
\le \int_D c\abs{\nabla \vb*{d}}^2
\le \frac{1}{2}\int_D \abs{\Delta c}\abs{\vb*{d}}^2-\frac{1}{a}\int_{D\setminus \Omega}c\abs{\vb*{d}}^2
\le \frac{1}{2}\int_D \abs{\Delta c}\abs{\vb*{d}}^2.
\]
The proof is complete.
\end{proof}

\section{Interval shape}

\subsection{The statements of the results}

\begin{theorem}[Interval shape in the whole line]
\label{t_int_whole}
Consider the situation
\[
\Omega := (f_l, f_r) \subset \bar{D} := \mathbb{R}
\]
with $f_l < f_r$.
Then, the followings hold.
\begin{enumerate}
\item
For each $a > 0$, the equation
\begin{equation}
\label{e_wf_int_whole}
a \int_{-\infty}^{+\infty} s' u'+\left(\int_{-\infty}^{+\infty}-\int_{f_l}^{f_r}\right) s u = \int_{f_l}^{f_r} u'
\quad \forall u \in H^1(\mathbb{R})
\end{equation}
admits a unique weak solution $s \in H^1(\mathbb{R})$.
\item
The solution $s$ can be calculated exactly with the exponential functions.
\item
The differential derivative $s'$ is constant in the shape domain $\Omega$.
\item
The PDE-based thickness $T_{\bar{D}, \Omega}^a := \frac{2}{\sqrt{a} s'}$ converges to the thickness $\bar{T}_{\Omega} = f_r-f_l$ on $\Omega$ as $a \to 0$.
\item
Moreover, the equality
\[
T_{\bar{D}, \Omega}^a = \bar{T}_\Omega+2 \sqrt{a}
\]
holds for all $a > 0$.
\end{enumerate}
\end{theorem}

\begin{theorem}[Interval shape in general interval]
\label{t_int_gen}
Consider the situation
\[
\Omega := (f_l, f_r) \subset D := (b_l, b_r)
\]
with $b_l < f_l < f_r < b_r$.
Then, the followings hold.
\begin{enumerate}
\item
For each $a > 0$, the equation
\begin{equation}
\label{e_wf_int_gen}
a \int_{b_l}^{b_r} s' u'+\left(\int_{b_l}^{b_r}-\int_{f_l}^{f_r}\right) s u = \int_{f_l}^{f_r} u'
\quad \forall u \in H_0^1(D)
\end{equation}
admits a unique weak solution $s \in H_0^1(D)$.
\item
The solution $s$ can be calculated exactly with the hyperbolic functions.
\item
The differential derivative $s'$ is constant in the shape domain $\Omega$.
\item
The PDE-based thickness $T_{D, \Omega}^a := \frac{2}{\sqrt{a} s'}$ converges to the thickness $\bar{T}_{\Omega} = f_r-f_l$ on $\Omega$ as $a \to 0$.
\item
Moreover, the estimate
\[
0 < 2\sqrt{a} \le T_{D, \Omega}^a-\bar{T}_\Omega \le 2\sqrt{a}+4 T\exp\qty(-2\frac{m}{\sqrt{a}})
\]
holds for all $a > 0$,
where $T := \bar{T}_\Omega = f_r-f_l > 0$ and $m := \min\{ b_r-f_r, f_l-b_l \} > 0$.
\end{enumerate}
\end{theorem}

\subsection{Proof of Theorem \ref{t_int_whole}}

The unique existence of the solution $s \in H^1(\mathbb{R})$ to \eqref{e_wf_int_whole} is guaranteed by the Lax-Milgram theorem as mentioned in the previous section.
First note that the strong form for the one dimensional problem \eqref{e_wf_int_whole} is given by
\[
\begin{cases}
-a s''+(1-\chi_\Omega)s = -\chi_\Omega' \quad \text{in $\mathbb{R}$,} \\
s(-\infty) = s(+\infty) = 0.
\end{cases}
\]
In particular, the solution $s$ satisfies
\[
s = a s'' \quad \text{in $(-\infty, f_l)\cup(f_r, +\infty)$,}
\]
\[
s'' = 0 \quad \text{in $(f_l, f_r)$.}
\]
Therefore, we see that $s$ has the form
\[
s(x) =
\begin{cases}
C_l \exp\left(\frac{x}{\sqrt{a}}\right) & \text{if $x \le f_l$,} \\
C_r \exp\left(-\frac{x}{\sqrt{a}}\right) & \text{if $x \ge f_r$.} \\
\end{cases}
\]
In the one dimensional space, since $H^1(I) \subset C^{0, 1/2}(I)$ for any bounded interval $I$,
we simply need to connect the two parts continuously with a linear expression.
In particular, the slope of $s$ in the shape domain $\Omega$ is constant, which is denoted by $p^*$.

We now claim
\begin{equation}
\label{e_int_jc}
a(s'(f_l+0)-s'(f_l-0)) = 1,
\quad a(s'(f_r-0)-s'(f_r+0)) = 1,
\end{equation}
where $s'(x+0)$ is a right derivative while $s'(x-0)$ is a left derivative.
Indeed, it follows from the weak form
\[
\begin{aligned}
&a \int_{-\infty}^{+\infty} s' u'+\left(\int_{-\infty}^{+\infty}-\int_{f_l}^{f_r}\right) s u
= a \int_{-\infty}^{+\infty} s' u'+a\left(\int_{-\infty}^{f_l}+\int_{f_r}^{+\infty}\right) s'' u \\
&\quad = a \int_{f_l}^{f_r} s' u'+a\left[s' u\right]_{-\infty}^{f_l}+a\left[s' u\right]_{f_r}^{+\infty} \\
&\quad = a\left[s' u\right]_{f_l}^{f_r}+a\left[s' u\right]_{-\infty}^{f_l}+a\left[s' u\right]_{f_r}^{+\infty} \\
&\quad = a(s'(f_l-0)-s'(f_l+0))u(f_l)+a(s'(f_r-0)-s'(f_r+0))u(f_r).
\end{aligned}
\]
Since it equals to
\[
\int_{f_l}^{f_r} u' = u(f_r)-u(f_l)
\]
for all test functions $u$,
we see that the equalities \eqref{e_int_jc} hold.

Due to the symmetry around $x = \frac{f_l+f_r}{2}$ in the setting,
it is enough to consider
\[
s(x) =
\begin{cases}
p^* \left(x-\frac{f_l+f_r}{2}\right) & \text{if $\frac{f_l+f_r}{2} \le x \le f_r$,} \\
C \exp\left(-\frac{x}{\sqrt{a}}\right) & \text{if $x \ge f_r$,} \\
\end{cases}
\]
with the conditions at $x = f_r$:
\[
p^* \frac{f_r-f_l}{2} = C \exp\left(-\frac{f_r}{\sqrt{a}}\right)
\]
and
\[
p^*+\frac{C}{\sqrt{a}} \exp\left(-\frac{f_r}{\sqrt{a}}\right) = \frac{1}{a}.
\]
Therefore, if $C$ is eliminated, it is calculated as follows:
\[
p^*
= \frac{1}{a} \left(1+\frac{1}{\sqrt{a}} \frac{f_r-f_l}{2}\right)^{-1}
= \frac{2}{\sqrt{a}((f_r-f_l)+2 \sqrt{a})}.
\]
The PDE-based thickness in this case is
\[
T_{\bar{D}, \Omega}^a = \frac{2}{\sqrt{a} p^*} = \bar{T}_\Omega+2 \sqrt{a}.
\]
This completes the proof of Theorem \ref{t_int_whole}.

\subsection{Proof of Theorem \ref{t_int_gen}}

We prove this theorem in a similar way to Theorem \ref{t_int_whole}.
First note that the strong form for the one dimensional problem \eqref{e_wf_int_gen} is
\[
\begin{cases}
-a s''+(1-\chi_\Omega)s = -\chi_\Omega' \quad \text{in $(b_l, b_r)$,} \\
s(b_l) = s(b_r) = 0.
\end{cases}
\]
In particular, the solution $s$ satisfies
\[
s = a s'' \quad \text{in $(b_l, f_l)\cup(f_r, b_r)$,}
\]
\[
s'' = 0 \quad \text{in $(f_l, f_r)$.}
\]
Hence, using integral constants $C_l$ and $C_r$, we have
\[
s(x) =
\begin{cases}
-C_l\sinh \frac{x-b_l}{\sqrt{a}} & \text{for $b_l \le x \le f_l$,} \\
-C_r\sinh \frac{x-b_r}{\sqrt{a}} & \text{for $f_r \le x \le b_r$.} \\
\end{cases}
\]
In the one-dimensional space, since $H^1(D) \subset C^{0, 1/2}(D)$, in $f_l \le x \le f_r$, we simply need to connect the two parts continuously with a linear expression.
The slope at this time is given by
\[
p^* := \frac{C_l\sinh\frac{f_l-b_l}{\sqrt{a}}+C_r\sinh\frac{b_r-f_r}{\sqrt{a}}}{T}.
\]

Since the equalities
\[
a(s'(f_l+0)-s'(f_l-0)) = 1,
\quad a(s'(f_r-0)-s'(f_r+0)) = 1
\]
hold as in the previous case,
$C_l$ and $C_r$ satisfy the simultaneous linear equations
\[
\mqty(\frac{a}{T}\sinh\frac{f_l-b_l}{\sqrt{a}}+\sqrt{a}\cosh\frac{f_l-b_l}{\sqrt{a}} & \frac{a}{T}\sinh\frac{b_r-f_r}{\sqrt{a}} \\ \frac{a}{T}\sinh\frac{f_l-b_l}{\sqrt{a}} & \frac{a}{T}\sinh\frac{b_r-f_r}{\sqrt{a}}+\sqrt{a}\cosh\frac{b_r-f_r}{\sqrt{a}})\mqty(C_l \\ C_r) = \mqty(1 \\ 1).
\]
For simplicity, let $\alpha := \frac{f_l-b_l}{\sqrt{a}}$, $\beta := \frac{b_r-f_r}{\sqrt{a}}$, and $k := \frac{T}{\sqrt{a}} = \frac{f_r-f_l}{\sqrt{a}}$.
We then have
\[
\mqty(\sinh\alpha+k\cosh\alpha & \sinh\beta \\ \sinh\alpha & \sinh\beta+k\cosh\beta)\mqty(C_l \\ C_r) = \frac{k}{\sqrt{a}}\mqty(1 \\ 1).
\]
This solves
\[
\mqty(C_l \\ C_r)
= \frac{1}{\sqrt{a}}\frac{k}{\sinh(\alpha+\beta)+k\cosh\alpha\cosh\beta}\mqty(\cosh\beta \\ \cosh\alpha).
\]
Therefore, as $a \to 0$ we got
\[
\sqrt{a}p^*
= \frac{1}{T}\frac{k\sinh(\alpha+\beta)}{\sinh(\alpha+\beta)+k\cosh\alpha\cosh\beta}
\to \frac{2}{T}.
\]

Moreover, we have the estimate on convergence speed as follows.
\[
T_{D, \Omega}-\bar{T}_\Omega
= T\qty(\frac{2}{k}+\frac{2\cosh\alpha\cosh\beta}{\sinh(\alpha+\beta)}-1)
= 2\sqrt{a}+\frac{1-\tanh\alpha+1-\tanh\beta}{\tanh\alpha+\tanh\beta}T
\]
Since $0 < \tanh\alpha, \tanh\beta < 1$, we see that
\[
\begin{aligned}
2\sqrt{a}
\le T_{D, \Omega}-\bar{T}_\Omega
&\le 2\sqrt{a}+\left(\frac{1-\tanh\alpha}{\tanh\alpha}+\frac{1-\tanh\beta}{\tanh\beta}\right)T \\
&\le 2\sqrt{a}+\left(2\exp(-2\alpha)+2\exp(-2\beta)\right)T \\
&\le 2\sqrt{a}+4T\exp\left(-2\frac{m}{\sqrt{a}}\right).
\end{aligned}
\]
This completes the proof of Theorem \ref{t_int_gen}.

\section{Band shape}

\subsection{The statements of the results}

\begin{theorem}[Band shape in the whole space]
\label{t_band_whole}
Consider the situation
\[
\Omega := \mathbb{R}\times(f_l, f_r) \subset \bar{D} := \mathbb{R}\times\mathbb{R}
\]
with $f_l < f_r$.
Then, the followings hold.
\begin{enumerate}
\item
For each $a > 0$ and $L > 0$, the equation
\begin{equation}
\label{e_wf_band_whole}
a \int_{\bar{D}_L} \nabla\vb*{s}:\nabla\vb*{u}+\int_{\bar{D}_L\setminus\Omega_L} \vb*{s}\cdot\vb*{u} = \int_{\Omega_L} \operatorname{div}\vb*{u}
\quad \forall \vb*{u} \in H^1((\mathbb{R}/L\mathbb{Z})\times\mathbb{R})^2.
\end{equation}
admits a unique solution $\vb*{s} \in H^1((\mathbb{R}/L\mathbb{Z})\times\mathbb{R})^2$.
Here, $H^1((\mathbb{R}/L\mathbb{Z})\times\mathbb{R})$ is the space of $L$-periodic functions in $x$ direction whose restriction to the unit region $\bar{D}_L := [0, L)\times\mathbb{R}$ belongs to $H^1(\mathbb{R}\times\mathbb{R})$
while $\Omega_L := [0, L)\times(f_l, f_r)$.
\item
The solution $\vb*{s}$ is of the form $\vb*{s}(x, y) = (0, S(y))$,
where $S$ is the solution of the one dimensional problem \eqref{e_wf_int_whole}.
\item
The divergence $\operatorname{div}\vb*{s}$ is constant in the shape domain $\Omega$.
\item
The PDE-based thickness $T_{\bar{D}, \Omega}^a := \frac{2}{\sqrt{a}\operatorname{div}\vb*{s}}$ converges to the thickness $\bar{T}_{\Omega} = f_r-f_l$ on $\Omega$ as $a \to 0$.
\item
Moreover, the equality
\[
T_{\bar{D}, \Omega}^a = \bar{T}_\Omega+2 \sqrt{a}
\]
holds for all $a > 0$.
\end{enumerate}
\end{theorem}

\begin{theorem}[Band shape in general domain]
\label{t_band_gen}
Consider the situation
\[
\Omega = \mathbb{R}\times(f_l, f_r)
\subset D = \{ (x, y) \in \mathbb{R}\times\mathbb{R} \mid b_l(x) < y < b_r(x) \}.
\]
with $\max b_l < f_l < f_r < \min b_r$.
Here, $b_l$ and $b_r$ are $C^1$ functions with a common periodicity $L > 0$ while $f_l$ and $f_r$ are constants.
Then, the followings hold.
\begin{enumerate}
\item
For each $a > 0$, the equation
\begin{equation}
\label{e_wf_band_gen}
a \int_{D_L} \nabla\vb*{s}:\nabla\vb*{u}+\int_{D_L\setminus\Omega_L} \vb*{s}\cdot\vb*{u} = \int_{\Omega_L} \operatorname{div}\vb*{u}
\quad \forall \vb*{u} \in H_0^1(D_{/L})^2.
\end{equation}
admits a unique solution $\vb*{s} \in H_0^1(D_{/L})^2$.
Here, $H_0^1(D_{/L})$ is the space of $L$-periodic functions in $x$ direction whose restriction to the unit region $D_L := D\cap([0, L)\times\mathbb{R})$ belongs to $H_0^1(D)$
while $\Omega_L := [0, L)\times(f_l, f_r)$.
\item
The solution $\vb*{s}$ is of the form $\vb*{s}(x, y) = (0, S(x, y))$,
where $S \in H_0^1(D_{/L})$ is the solution of the scalar-valued equation
\[
a\int_{D_L} \nabla S\cdot\nabla u+\int_{D_L\setminus\Omega_L} S u = \int_{\Omega_L} u_y
\quad \forall u \in H_0^1(D_{/L}).
\]
\item
The PDE-based thickness $T_{D, \Omega}^a := \frac{2}{\sqrt{a}\operatorname{div}\vb*{s}}$ converges to the thickness $\bar{T}_{\Omega} = f_r-f_l$ a.e.\ on $\Omega$ as $a \to 0$.
\item
Moreover, the estimate
\[
\norm{\frac{1}{T_{D, \Omega}^a}-\frac{1}{\bar{T}_{\Omega}}}_{L^2(\Omega_L)}
\le 2\frac{\sqrt{\abs{\Omega_L}}}{T^2} \sqrt{a}+2\frac{\sqrt{L}}{\sqrt{m}} \exp\left(-\frac{m}{\sqrt{a}}\right)
\]
holds for all $a > 0$.
Here, $T := \bar{T}_\Omega = f_r-f_l > 0$, $\abs{\Omega_L}$ is the area of $\Omega_L$, i.e.\ $\abs{\Omega_L} := L T > 0$ and $m := \min\{ \min b_r-f_r, f_l-\max b_l \} > 0$ .
\end{enumerate}
\end{theorem}

\subsection{Proof of Theorem \ref{t_band_whole}}

For $\vb*{u} = (u^x, u^y) \in H^1((\mathbb{R}/L\mathbb{Z})\times\mathbb{R})^2$,
because of the periodicity in $x$ direction, we have
\[
\int_{\Omega_L} \operatorname{div}\vb*{u}
= \int_{\Omega_L} (u^x_x+u^y_y)
= \int_{\Omega_L} u^y_y.
\]
Thus, the solution $\vb*{s} = (s^x, s^y) \in H^1((\mathbb{R}/L\mathbb{Z})\times\mathbb{R})^2$ solves
\[
a\int_{\bar{D}_L} \nabla s^x\cdot\nabla u^x+\int_{\bar{D}_L\setminus\Omega_L} s^x u^x = 0
\quad \forall u^x \in H^1((\mathbb{R}/L\mathbb{Z})\times\mathbb{R})
\]
and
\[
a\int_{\bar{D}_L} \nabla s^y\cdot\nabla u^y+\int_{\bar{D}_L\setminus\Omega_L} s^y u^y = \int_{\Omega_L} u^y_y
\quad \forall u^y \in H^1((\mathbb{R}/L\mathbb{Z})\times\mathbb{R}).
\]
Since $s^x \in H^1((\mathbb{R}/L\mathbb{Z})\times\mathbb{R})$ is a weak solution of the homogeneous equation, it is identically zero in $\mathbb{R}\times\mathbb{R}$.
Therefore, we have $\vb*{s}(x, y) = (0, S(x, y))$ with the unique weak solution $S$ of the scalar-valued equation
\begin{equation}
\label{e_wf_band_whole_sca}
a\int_{\bar{D}_L} \nabla S\cdot\nabla u+\int_{\bar{D}_L\setminus\Omega_L} S u = \int_{\Omega_L} u_y
\quad \forall u \in H^1((\mathbb{R}/L\mathbb{Z})\times\mathbb{R}).
\end{equation}

Moreover, taking the solution $s$ of the one dimensional problem \eqref{e_wf_int_whole} and extending it to $D$ as $\tilde{S}(x, y) = s(y)$,
we obtain the another solution of the scalar-valued equation \eqref{e_wf_band_whole_sca}.
Since the solution is unique, we see that $S(x, y) = s(y)$.

The remains are to be shown from the one dimensional problem (Theorem \ref{t_int_whole}).

\subsection{Proof of Theorem \ref{t_band_gen}}

We first remark that by repeating the argument in the proof of Theorem \ref{t_band_whole}, the assertion (1) and (2) hold.

If we consider solving the equation \eqref{e_wf_band_whole} in the whole space
\[
\bar{D} = \mathbb{R}\times\mathbb{R}
\]
and then restricting it to the desired domain $D$.
As calculated in the previous section, the solution $\bar{\vb*{s}}$ is given by $\bar{\vb*{s}} = (0, \bar{S}(y))$ with
\[
\bar{S}(y) =
\begin{cases}
-\frac{2 T}{\sqrt{a}(T+2\sqrt{a})} \exp\left(\frac{y-f_l}{\sqrt{a}}\right) & \text{if $y \le f_l$,} \\
\frac{2 T}{\sqrt{a}(T+2\sqrt{a})} \exp\left(-\frac{y-f_r}{\sqrt{a}}\right) & \text{if $y \ge f_r$.} \\
\end{cases}
\]
Hence, we have
\[
\begin{aligned}
\norm{\bar{\vb*{s}}}_{L^\infty(\partial D)}
&= \max\{ \abs{\bar{S}(\max b_l)}, \abs{\bar{S}(\min b_r)} \} \\
&\le \frac{2 T}{\sqrt{a}(T+2\sqrt{a})} \exp\left(-\frac{m}{\sqrt{a}}\right)
\le \frac{2}{\sqrt{a}} \exp\left(-\frac{m}{\sqrt{a}}\right).
\end{aligned}
\]
Moreover, by the maximum principle (Lemma \ref{t_mp}), we see that
\[
\norm{\vb*{s}-\bar{\vb*{s}}}_{L^\infty(D)}
\le \norm{\vb*{s}-\bar{\vb*{s}}}_{L^\infty(\partial D)}
= \norm{\bar{\vb*{s}}}_{L^\infty(\partial D)}
\le \frac{2}{\sqrt{a}} \exp\left(-\frac{m}{\sqrt{a}}\right).
\]

We now claim the following interior $H^1$ estimate for band shape.

\begin{lemma}[Interior $H^1$ estimate for band shape]
The inequality
\[
\int_{\Omega_L} \abs{\nabla \vb*{d}}^2 \le 2 L\left(\frac{1}{f_l-\max b_l}+\frac{1}{\min b_r-f_r}\right) \norm{\vb*{d}}_{L^\infty(D)}^2
\]
holds.
\end{lemma}

\begin{proof}
Note that for $m > 0$ the function
\[
y =
\begin{cases}
0 & (x \le 0), \\
\frac{2}{m^2}x^2 & (0 \le x \le m/2), \\
1-\frac{2}{m^2}(m-x)^2 & (m/2 \le x \le m), \\
1 & (x \ge m), \\
\end{cases}
\]
is a $W^{2, \infty} = C^{1, 1}$ function which satisfies
\[
|y''| \le \frac{4}{m^2}.
\]
This gives us existence of a cutoff function $c \in W^{2, \infty}(\mathbb{R})$ that satisfies
\[
0 \le c \le 1,
\quad \text{$c = 1$ on $[f_l, f_r]$,}
\quad \text{$c = 0$ on $(-\infty, \max b_l]\cup[\min b_r, \infty)$,}
\]
\[
\text{$|c''| \le \frac{4}{(f_l-\max b_l)^2}$ in $(\max b_l, f_l)$,}
\quad \text{$|c''| \le \frac{4}{(\min b_r-f_r)^2}$ in $(f_r, \min b_r)$.}
\]
Therefore, applying Lemma \ref{t_ire_ab} with this cutoff function $c$, we see that is lemma holds true.
\end{proof}

By using this lemma, it follows that
\[
\int_{\Omega_L} \abs{\operatorname{div} (\vb*{s}-\bar{\vb*{s}})}^2
= \int_{\Omega_L} \abs{(S-\bar{S})_y}^2
\le \int_{\Omega_L} \abs{\nabla (\vb*{s}-\bar{\vb*{s}})}^2
\le \frac{4 L}{m} \left(\frac{2}{\sqrt{a}} \exp\left(-\frac{m}{\sqrt{a}}\right)\right)^2
\]
and hence
\[
\int_{\Omega_L} \abs{\frac{1}{T_{D, \Omega}^a}-\frac{1}{T_{\bar{D}, \Omega}^a}}^2
\le \frac{4 L}{m} \exp\left(-\frac{2 m}{\sqrt{a}}\right).
\]
Therefore, we finally obtain the estimate
\[
\begin{aligned}
\norm{\frac{1}{T_{D, \Omega}^a}-\frac{1}{\bar{T}_{\Omega}}}_{L^2(\Omega_L)}
&\le \norm{\frac{1}{T_{D, \Omega}^a}-\frac{1}{T_{\bar{D}, \Omega}^a}}_{L^2(\Omega_L)}+\norm{\frac{1}{T_{\bar{D}, \Omega}^a}-\frac{1}{\bar{T}_{\Omega}}}_{L^2(\Omega_L)} \\
&\le \frac{1}{T^2}\left(\int_{\Omega_L} \abs{T_{\bar{D}, \Omega}^a-\bar{T}_{\Omega}}^2\right)^{\frac{1}{2}}+\left(\int_{\Omega_L} \abs{\frac{1}{T_{D, \Omega}^a}-\frac{1}{T_{\bar{D}, \Omega}^a}}^2\right)^{\frac{1}{2}} \\
&\le 2\frac{\sqrt{\abs{\Omega_L}}}{T^2} \sqrt{a}+2\sqrt{\frac{L}{m}} \exp\left(-\frac{m}{\sqrt{a}}\right).
\end{aligned}
\]
This completes the proof of Theorem \ref{t_band_gen}.

\section{Annulus shape}

\subsection{The statements of the results}

\begin{theorem}[Annulus shape in the whole space]
\label{t_anu_whole}
Consider the situation
\[
\Omega := \{ (x, y) \in \mathbb{R}^2 \mid f_l^2 < x^2+y^2 < f_r^2 \} \subset \bar{D} := \mathbb{R}^2
\]
with $0 < f_l < f_r$.
Then, the followings hold.
\begin{enumerate}
\item
For each $a > 0$, the equation
\begin{equation}
\label{e_wf_annulus_whole}
a \int_D \nabla\vb*{s}:\nabla\vb*{u}+\int_{D\setminus\Omega} \vb*{s}\cdot\vb*{u} = \int_\Omega \operatorname{div}\vb*{u}
\quad \forall \vb*{u} \in H^1(D)^2.
\end{equation}
admits a unique solution $\vb*{s} \in H^1(D)^2$.
\item
The solution is of the form $\vb*{s} = (S(r)\cos\theta, S(r)\sin\theta)$,
where $S(r)$ is expressed in terms of the modified Bessel functions $I_n$ and $K_n$.
\item
The divergence $\operatorname{div}\vb*{s}$ is constant in the shape domain $\Omega$.
\item
The PDE-based thickness $T_{D, \Omega}^a := \frac{2}{\sqrt{a}\operatorname{div}\vb*{s}}$ converges to the thickness $\bar{T}_{\Omega} = f_r-f_l$ on $\Omega$ as $a \to 0$.
\item
Moreover, the estimate
\[
0
< \frac{3 f_r+f_l}{2 f_r}\sqrt{a}
\le T_{D, \Omega}^a-\bar{T}_{\Omega}
\le 2\frac{f_r}{f_l}\sqrt{a}
\]
holds for all $a > 0$.
\end{enumerate}
\end{theorem}

\begin{theorem}[Annulus shape in general domain]
\label{t_anu_gen}
Consider the situation
\[
\Omega := \{ (x, y) \in \mathbb{R}^2 \mid f_l^2 < x^2+y^2 < f_r^2 \} \subset D
\]
with $0 < f_l < f_r < b_r := \sup\{ b > 0 \mid \{ (x, y) \in \mathbb{R}^2 \mid x^2+y^2 < b^2 \} \subset D \}$.
Then, the followings hold.
\begin{enumerate}
\item
For each $a > 0$, the equation
\begin{equation}
\label{e_wf_annulus_gen}
a \int_D \nabla\vb*{s}:\nabla\vb*{u}+\int_{D\setminus\Omega} \vb*{s}\cdot\vb*{u} = \int_\Omega \operatorname{div}\vb*{u}
\quad \forall \vb*{u} \in H_0^1(D)^2.
\end{equation}
admits a unique solution $\vb*{s} \in H_0^1(D)^2$.
\item
The PDE-based thickness $T_{D, \Omega}^a := \frac{2}{\sqrt{a}\operatorname{div}\vb*{s}}$ converges to the thickness $\bar{T}_{\Omega} = f_r-f_l$ a.e.\ on $\Omega$ as $a \to 0$.
\item
Moreover, the estimate
\[
\norm{\frac{1}{T_{D, \Omega}^a}-\frac{1}{\bar{T}_{\Omega}}}_{L^2(\Omega)}
\le 2\frac{f_r}{f_l}\frac{\sqrt{\abs{\Omega}}}{T^2} \sqrt{a}+2\sqrt{\pi}\frac{\sqrt{f_r}}{\sqrt{m}}\exp\left(-\frac{m}{\sqrt{a}}\right)
\]
holds for all $a > 0$.
Here, $T := \bar{T}_{\Omega} = f_r-f_l$, $\abs{\Omega}$ is the area of $\Omega$, i.e. $\abs{\Omega} := \pi(f_r^2-f_l^2)$ and $m := b_r-f_r > 0$.
\end{enumerate}
\end{theorem}

\subsection{Proof of Theorem \ref{t_anu_whole}}

By the polar coordinate transformation
\[
(x, y) = \Phi(r, \theta) = (r \cos \theta, r \sin \theta)
\]
we discuss the domain 
\[
\tilde{\Omega} := (f_l, f_r)\times [0, 2\pi)
\subset \tilde{D} := (0, \infty)\times [0, 2\pi)
\]
on the $r \theta$ plane.

Calculating the Jacobian matrix, we have
\[
D\Phi
= \begin{pmatrix}x_r & x_\theta \\ y_r & y_\theta\end{pmatrix}
= \begin{pmatrix}\cos\theta & -r\sin\theta \\ \sin\theta & r\cos\theta\end{pmatrix}
= R(\theta)\begin{pmatrix}1 & 0 \\ 0 & r\end{pmatrix},
\]
where $R(\theta) := \begin{pmatrix}\cos\theta & -\sin\theta \\ \sin\theta & \cos\theta\end{pmatrix}$ is the rotation matrix.
The inverse is
\[
A
:= D\Phi^{-1}
= \frac{1}{r}\begin{pmatrix}r & 0 \\ 0 & 1\end{pmatrix}R(-\theta),
\]
\[
\det A^{-1} = r,
\]
\[
A A^T = \frac{1}{r^2}\begin{pmatrix}r^2 & 0 \\ 0 & 1\end{pmatrix}.
\]
For $\vb*{u}(x, y) = \vb*{u}(\Phi(r, \theta)) = \vb*{U}(r, \theta)$, we have
\[
\nabla_{(x, y)}\vb*{u} = \nabla_{(r, \theta)}\vb*{U} A.
\]
Let $U^r$ and $U^\theta$ be a normal and tangential component of the vector field $\vb*{U}$ in the polar coordinate system, respectively,
that is,
\[
\vb*{U} = U^r\begin{pmatrix}\cos\theta \\ \sin\theta\end{pmatrix}+U^\theta\begin{pmatrix}-\sin\theta \\ \cos\theta\end{pmatrix}
= R(\theta)\begin{pmatrix}U^r \\ U^\theta\end{pmatrix}.
\]
Then, the partial derivatives are
\[
\vb*{U}_r = R(\theta)\begin{pmatrix}U^r_r \\ U^\theta_r\end{pmatrix},
\quad \vb*{U}_\theta = R(\theta)\begin{pmatrix}U^r_\theta \\ U^\theta_\theta\end{pmatrix}+R'(\theta)\begin{pmatrix}U^r \\ U^\theta\end{pmatrix}.
\]
Let $R := R(\frac{\pi}{2}) = \begin{pmatrix}0 & -1 \\ 1 & 0\end{pmatrix}$, then note that $R'(\theta) = R(\theta+\frac{\pi}{2}) = R(\theta)R$.
From here,
\[
\begin{aligned}
\operatorname{div}\vb*{u}
= \tr \nabla \vb*{U} A
&= \frac{1}{r} \tr \begin{pmatrix}r & 0 \\ 0 & 1\end{pmatrix}R(-\theta)\left[R(\theta)\begin{pmatrix}U^r_r & U^r_\theta \\ U^\theta_r & U^\theta_\theta\end{pmatrix}+R'(\theta)\begin{pmatrix}0 & U^r \\ 0 & U^\theta\end{pmatrix}\right] \\
&= \frac{1}{r} \tr \begin{pmatrix}r & 0 \\ 0 & 1\end{pmatrix}\left[\begin{pmatrix}U^r_r & U^r_\theta \\ U^\theta_r & U^\theta_\theta\end{pmatrix}+R\begin{pmatrix}0 & U^r \\ 0 & U^\theta\end{pmatrix}\right] \\
&= \frac{1}{r} (r U^r_r+U^\theta_\theta+U^r),
\end{aligned}
\]
\[
\int_\Omega \operatorname{div}\vb*{u}
= \int_{\tilde{\Omega}} \tr \nabla \vb*{U} A \det A^{-1}
= \int_{\tilde{\Omega}} (r U^r_r+U^\theta_\theta+U^r)
= \int_{\tilde{\Omega}} (r U^r_r+U^r).
\]
Here, we have used the periodicity in $\theta$ direction.
Moreover,
\[
\begin{aligned}
\int_D \nabla \vb*{s}:\nabla \vb*{u}
&= \int_{\tilde{D}} \nabla \vb*{S} A:\nabla \vb*{U} A \det A^{-1}
= \int_{\tilde{D}} \nabla \vb*{S}:\nabla \vb*{U} A A^T \det A^{-1} \\
&= \int_{\tilde{D}} \frac{1}{r} \begin{pmatrix}\vb*{S}_r & \vb*{S}_\theta\end{pmatrix}:\begin{pmatrix}\vb*{U}_r & \vb*{U}_\theta\end{pmatrix}\begin{pmatrix}r^2 & 0 \\ 0 & 1\end{pmatrix}
= \int_{\tilde{D}} \frac{1}{r} (r^2\vb*{S}_r\cdot\vb*{U}_r+\vb*{S}_\theta\cdot\vb*{U}_\theta) \\
&= \int_{\tilde{D}} \left[r\begin{pmatrix}S^r_r \\ S^\theta_r\end{pmatrix}\cdot\begin{pmatrix}U^r_r \\ U^\theta_r\end{pmatrix}+\frac{1}{r}\begin{pmatrix}S^r_\theta-S^\theta \\ S^\theta_\theta+S^r\end{pmatrix}\cdot\begin{pmatrix}U^r_\theta-U^\theta \\ U^\theta_\theta+U^r\end{pmatrix}\right].
\end{aligned}
\]
Therefore, the weak form \eqref{e_wf_annulus_whole} becomes
\begin{equation}
\label{e_wf_annulus_whole_polar}
\begin{aligned}
&a \int_{\tilde{D}} \left[r\begin{pmatrix}S^r_r \\ S^\theta_r\end{pmatrix}\cdot\begin{pmatrix}U^r_r \\ U^\theta_r\end{pmatrix}+\frac{1}{r}\begin{pmatrix}S^r_\theta-S^\theta \\ S^\theta_\theta+S^r\end{pmatrix}\cdot\begin{pmatrix}U^r_\theta-U^\theta \\ U^\theta_\theta+U^r\end{pmatrix}\right]
+\int_{\tilde{D}\setminus\tilde{\Omega}} r\begin{pmatrix}S^r \\ S^\theta\end{pmatrix}\cdot\begin{pmatrix}U^r \\ U^\theta\end{pmatrix} \\
&= \int_{\tilde{\Omega}} (r U^r_r+U^r)
\quad \forall (U^r, U^\theta) \in H_\Phi.
\end{aligned}
\end{equation}
Here, $(U^r, U^\theta) \in H_\Phi$ means that $U^r$ and $U^\theta$ are $2\pi$-periodic functions satisfying
\[
\norm{(U^r, U^\theta)}_{H_\Phi, a}^2
:= a \int_{\tilde{D}} \left[r\begin{vmatrix}U^r_r \\ U^\theta_r\end{vmatrix}^2+\frac{1}{r}\begin{vmatrix}U^r_\theta-U^\theta \\ U^\theta_\theta+U^r\end{vmatrix}^2\right]+\int_{\tilde{D}} r\begin{vmatrix}U^r \\ U^\theta\end{vmatrix}^2 < \infty.
\]

From the equation \eqref{e_wf_annulus_whole_polar}, we would like to show that $S^\theta \equiv 0$ and $S^r_\theta \equiv 0$ are true.
If this is the case, then $S^r = S^r(r)$ is a solution to
\begin{equation}
\label{e_wf_annulus_whole_radial}
a \int_{\tilde{D}} \left[r S^r_r U^r_r +\frac{1}{r} S^r U^r\right]
+\int_{\tilde{D}\setminus\tilde{\Omega}} r S^r U^r
= \int_{\tilde{\Omega}} (r U^r_r+U^r)
\quad \forall U^r \in \tilde{H},
\end{equation}
where $\tilde{H}$ is the space of $2\pi$-periodic functions $U^r$ satisfying
\[
\norm{U^r}_{\tilde{H}, a}^2
:= a \int_{\tilde{D}} \left[r \abs{U^r_r}^2+\frac{1}{r}\abs{U^r}^2\right]+\int_{\tilde{D}} r\abs{U^r}^2 < \infty.
\]
Note that the solution $S^r \in \tilde{H}$ to the equation \eqref{e_wf_annulus_whole_radial} exists since Lax-Milgram theorem is available.
We now consider $(S^r, S^\theta) = (S^r(r), 0)$.
Then, one can easily check the conditions $(S^r, S^\theta) \in H_\Phi$ and \eqref{e_wf_annulus_whole_polar}
since
\[
\int_{\tilde{D}} \frac{1}{r} S^r U^\theta_\theta = 0
\]
thanks to the periodicity in $\theta$ direction.
Thus, $(S^r(r), 0)$ is a solution, and due to the uniqueness of the solutions of \eqref{e_wf_annulus_whole_polar}, we have $S^\theta = 0$ and $S^r_\theta = 0$.

Therefore, what remains is to analyze the weak form $S = S^r(r)$ satisfies
\[
a \int_0^\infty \left[r S_r U_r+\frac{1}{r} S U\right]+\left(\int_0^\infty-\int_{f_l}^{f_r}\right) r S U
= \int_{f_l}^{f_r} (r U^r_r+U^r)
\quad \forall U \in H,
\]
where $H$ is the space of functions $U$ satisfying
\[
\sqrt{r}U(r), \frac{1}{\sqrt{r}}U(r), \sqrt{r}U_r(r) \in L^2(0, \infty).
\]
Here, if we convert $u(r) = r U(r)$, we get
\[
a \int_0^\infty \left[S_r u_r-\frac{1}{r}S_r u+\frac{1}{r^2} S u\right]+\left(\int_0^\infty-\int_{f_l}^{f_r}\right) S u
= \int_{f_l}^{f_r} u_r
= [u]_{r = f_l}^{f_r}.
\]

This can be written in the strong form as
\[
-S_{r r}-\frac{1}{r} S_r+\frac{1}{r^2} S+\frac{1}{a}(1-\chi (r)) S = \frac{1}{a} \chi' (r)
\]
and in particular, in the interior of the void domain $(0, f_l) \cup (f_r, \infty)$, it is the modified Bessel equation
\[
S_{r r}+\frac{1}{r} S_r-\left(\frac{1}{a}+\frac{1}{r^2}\right)S = 0.
\]
Therefore, using the first kind of modified Bessel function $I_n (x)$ and the second kind of modified Bessel function $K_n (x)$,
we have
\[
S(r) = C I_1\left(\frac{r}{\sqrt{a}}\right) \quad \text{in $(0, f_l)$,}
\quad S(r) = D K_1\left(\frac{r}{\sqrt{a}}\right) \quad \text{in $(f_r, \infty)$.}
\]
Here, $C$ and $D$ are constants to be determined later.
We refer the readers to \cite{r_as64} for the details of the modified Bessel functions.

Also, set
\[
p(r) := \frac{1}{r}(r S_r+S).
\]
Then, we can rewrite the weak form as
\[
a \int_0^\infty p u_r+\left(\int_0^\infty-\int_{f_l}^{f_r}\right) S u = \int_{f_l}^{f_r} u_r
\]
In particular, for $u \in H_0^1(f_l, f_r)$, since $U(r) = \frac{1}{r}u(r) \in H$, we have
\[
\int_{f_l}^{f_r} p u_r = 0
\]
and hence $p$ is a constant function on $(f_l, f_r)$.
Setting this value as $p^*$,
we have
\[
S(r) = \frac{1}{2} p^* r+\frac{E}{r},
\]
where $E$ is a constant.
From this,
\[
p^*
= 2\frac{f_r S(f_r)-f_l S(f_l)}{f_r^2-f_l^2}
\]
can be obtained.

Furthermore, $p(r) = S_r(r)+\frac{1}{r}S(r)$ is given by
\[
p(r) = \frac{C}{\sqrt{a}} I_0\left(\frac{r}{\sqrt{a}}\right) \quad \text{in $(0, f_l)$,}
\quad p(r) = -\frac{D}{\sqrt{a}} K_0\left(\frac{r}{\sqrt{a}}\right) \quad \text{in $(f_r, \infty)$.}
\]
Here, we have used the relations $I_1'(x) = I_0(x)-\frac{1}{x}I_1(x)$ and $K_1'(x) = -K_0(x)-\frac{1}{x}K_1(x)$ introduced in \cite[9.6.26]{r_as64}.
We also have
\[
p(r) = p^* = 2\frac{D f_r K_1 (\frac{f_r}{\sqrt{a}})-C f_l I_1 (\frac{f_l}{\sqrt{a}})}{f_r^2-f_l^2} \quad \text{in $(f_l, f_r)$.}
\]

All that remains is the equation that should be satisfied at $r = f_l, f_r$.
This is given by
\[
a(p^*-p(f_l-0)) = 1,
\quad a(p^*-p(f_r+0)) = 1.
\]
So we have the simultaneous linear equations
\[
\begin{pmatrix}
k(f_r+f_l)I_0+2 f_l I_1 & -2 f_r K_1 \\
-2 f_l I_1 & k(f_r+f_l)K_0+2 f_r K_1 \\
\end{pmatrix}
\begin{pmatrix}C \\ D\end{pmatrix}
= \frac{k(f_r+f_l)}{\sqrt{a}}\begin{pmatrix}-1 \\ 1\end{pmatrix},
\]
where $k = \frac{T}{\sqrt{a}} = \frac{f_r-f_l}{\sqrt{a}}$.
We also have suppressed the argument of the modified Bessel functions for simplicity,
that is, $I_n := I_n(\frac{f_l}{\sqrt{a}})$ and $K_n := K_n(\frac{f_r}{\sqrt{a}})$.
Solving this, the determinant is
\[
d = k^2 (f_r+f_l)^2 I_0 K_0+2 k(f_r+f_l)(f_l I_1 K_0+f_r I_0 K_1)
\]
and
\[
\begin{aligned}
\begin{pmatrix}C \\ D\end{pmatrix}
&= \frac{1}{d}
\begin{pmatrix}
k(f_r+f_l)K_0+2 f_r K_1 & 2 f_r K_1 \\
2 f_l I_1 & k(f_r+f_l)I_0+2 f_l I_1 \\
\end{pmatrix}
\frac{k(f_r+f_l)}{\sqrt{a}} \begin{pmatrix}-1 \\ 1\end{pmatrix} \\
&= \frac{k^2 (f_r+f_l)^2}{\sqrt{a} d}\begin{pmatrix}-K_0 \\ I_0\end{pmatrix}.
\end{aligned}
\]
From here
\[
\begin{aligned}
p^*
&= 2 \frac{D f_r K_1-C f_l I_1}{f_r^2-f_l^2} \\
&= \frac{2}{\sqrt{a}} \frac{k(f_r I_0 K_1+f_l I_1 K_0)}{k(f_r+f_l)I_0 K_0+2(f_l I_1 K_0+f_r I_0 K_1)} \frac{1}{f_r-f_l}.
\end{aligned}
\]

Now, we claim the convergence
\[
\frac{1}{T_{D, \Omega}^a}
= \frac{\sqrt{a} \operatorname{div}\vb*{s}}{2}
= \frac{\sqrt{a} p^*}{2}
\to \frac{1}{f_r-f_l}
= \frac{1}{\bar{T}_{\Omega}}.
\]
Indeed, it is easy to see that
\[
\begin{aligned}
T_{D, \Omega}^a-\bar{T}_{\Omega}
&= \left(\frac{2}{k}+\frac{(f_r+f_l)I_0 K_0}{f_r I_0 K_1+f_l I_1 K_0}-1\right)T \\
&= 2\sqrt{a}+\frac{f_r I_0 (K_0-K_1)+f_l (I_0-I_1) K_0}{f_r I_0 K_1+f_l I_1 K_0} T.
\end{aligned}
\]
Since the inequalities $I_0 (x) \ge I_1 (x) > 0$ and $0 < K_0 (x) \le K_1 (x)$ hold for $x > 0$ in general,
we have
\[
2\sqrt{a}+\frac{K_0-K_1}{K_1} T
\le T_{D, \Omega}^a-\bar{T}_{\Omega}
\le 2\sqrt{a}+\frac{I_0-I_1}{I_1} T.
\]

In order to estimate the left- and right-hand sides, we recall the following inequalities for the modified Bessel functions of Segura-type \cite{r_s23}.

\begin{proposition}[{\cite[Theorem 1 with $\nu = 1/2$]{r_s11}}]
\label{t_besselk}
For the second kind of modified Bessel functions $K_n$, the inequality
\[
\frac{K_0(x)}{K_1(x)} \ge \frac{x}{\frac{1}{2}+\sqrt{\frac{1}{4}+x^2}}
\]
holds for any $x > 0$.
\end{proposition}

\begin{proposition}[{\cite[Theorem 8 with $\nu = 1$ and $\lambda = 1/2$]{r_s23}}]
\label{t_besseli}
For the first kind of modified Bessel functions $I_n$, the inequality
\[
\frac{I_0(x)}{I_1(x)} \le \frac{\frac{1}{2}+\sqrt{\frac{9}{4}+x^2}}{x}
\]
holds for any $x > 0$.
\end{proposition}

It follows from Proposition \ref{t_besselk} that
\begin{equation}
\label{e_besselk}
\frac{K_0(x)-K_1(x)}{K_1(x)}
\ge \frac{x}{\frac{1}{2}+\sqrt{\frac{1}{4}+x^2}}-1
= \frac{\sqrt{\frac{1}{4}+x^2}-\frac{1}{2}}{x}-1
\ge -\frac{1}{2 x}
\end{equation}
and hence
\[
\begin{aligned}
T_{D, \Omega}^a-\bar{T}_{\Omega}
&\ge 2\sqrt{a}+\frac{K_0\left(\frac{f_r}{\sqrt{a}}\right)-K_1\left(\frac{f_r}{\sqrt{a}}\right)}{K_1\left(\frac{f_r}{\sqrt{a}}\right)}(f_r-f_l) \\
&\ge 2\sqrt{a}-\frac{f_r-f_l}{2 f_r}\sqrt{a}
= \frac{3 f_r+f_l}{2 f_r}\sqrt{a}
> 0.
\end{aligned}
\]

It follows from Proposition \ref{t_besseli} that
\begin{equation}
\label{e_besseli}
\frac{I_0(x)-I_1(x)}{I_1(x)}
\le \frac{\frac{1}{2}+\sqrt{\frac{9}{4}+x^2}}{x}-1
\le \frac{2}{x}
\end{equation}
and hence
\[
\begin{aligned}
T_{D, \Omega}^a-\bar{T}_{\Omega}
&\le 2\sqrt{a}+\frac{I_0\left(\frac{f_l}{\sqrt{a}}\right)-I_1\left(\frac{f_l}{\sqrt{a}}\right)}{I_1\left(\frac{f_l}{\sqrt{a}}\right)}(f_r-f_l) \\
&\le 2\sqrt{a}+\frac{2(f_r-f_l)}{f_l}\sqrt{a}
= 2\frac{f_r}{f_l}\sqrt{a}.
\end{aligned}
\]

Therefore, we have proved all the statements in Theorem \ref{t_anu_whole}.

\begin{remark}
There is another proof of \eqref{e_besseli} without using Proposition \ref{t_besseli}.
Indeed, according to \cite[9.6.26]{r_as64}, we have
\[
I_2(x) = I_0(x)-\frac{2}{x}I_1(x).
\]
So the inequality to be shown is $\frac{I_2(x)}{I_1(x)} \le 1$.
This is known to be true by \cite{r_n74} and the references cited in it.
\end{remark}

\subsection{Proof of Theorem \ref{t_anu_gen}}

We consider solving the equation in the whole space
\[
\bar{D} = \mathbb{R}^2
\]
and then restricting it to the desired domain $D$.
As calculated in the previous subsection, the solution $\bar{\vb*{s}}$ is given by $\bar{\vb*{s}} = (\bar{S}(r)\cos\theta, \bar{S}(r)\sin\theta)$ with
\[
\bar{S}(r)
= D K_1\left(\frac{r}{\sqrt{a}}\right)
= \frac{1}{\sqrt{a}} \frac{k(f_r+f_l) I_0 K_1\left(\frac{r}{\sqrt{a}}\right)}{k (f_r+f_l) I_0 K_0+2 (f_l I_1 K_0+f_r I_0 K_1)}.
\]
for $r \ge b_r$ with $I_n := I_n(\frac{f_l}{\sqrt{a}})$ and $K_n := K_n(\frac{f_r}{\sqrt{a}})$.
Since $K_1$ is monotonically decreasing,
\[
\norm{\bar{\vb*{s}}}_{L^\infty(\partial D)}
= \bar{S}(b_r)
\le \frac{1}{\sqrt{a}} \frac{k(f_r+f_l) I_0 K_1\left(\frac{b_r}{\sqrt{a}}\right)}{k (f_r+f_l) I_0 K_0+2 (f_l I_1 K_0+f_r I_0 K_1)}.
\]
Set the right-hand side of this to $M_a$.
By the maximum principle (Lemma \ref{t_mp}), we see that
\[
\norm{\vb*{s}-\bar{\vb*{s}}}_{L^\infty(D)}
\le \norm{\vb*{s}-\bar{\vb*{s}}}_{L^\infty(\partial D)}
= \norm{\bar{\vb*{s}}}_{L^\infty(\partial D)}
\le M_a.
\]

The interior $H^1$ estimate to be used in this section is as follows.

\begin{lemma}[Interior $H^1$ estimate for annulus shape]
The inequality
\[
\int_\Omega \abs*{\nabla \vb*{d}}^2 \le \frac{4\pi(b_r^2+f_r^2)}{b_r^2-f_r^2} \norm{\vb*{d}}_{L^\infty(D)}^2.
\]
holds.
\end{lemma}

\begin{proof}
Let $f := f_r$ and $b := b_r$ for simplicity.
Construct a function $c \in C^{1, 1}(\mathbb{R}^2)$ satisfying
\[
\begin{cases}
c(x, y) = 1 & \text{for $\sqrt{x^2+y^2} \le f$,} \\
\Delta c(x, y) = -K & \text{for $f < \sqrt{x^2+y^2} < p$,} \\
\Delta c(x, y) = K & \text{for $p < \sqrt{x^2+y^2} < b$,} \\
c(x, y) = 0 & \text{for $\sqrt{x^2+y^2} \ge b$,} \\
\end{cases}
\]
for some constant $K \ge 0$ and a point $p \in [f, b]$
to apply Lemma \ref{t_ire_ab}.
Note apparently that such $c$ is radially symmetric.
Since $\Delta c = c_{r r}+\frac{1}{r}c_r$ with $r = \sqrt{x^2+y^2}$
and the general solution of the ordinary differential equation $c_{r r}+\frac{1}{r}c_r = K$ is given by $c(r) = \frac{K}{4}r^2+k_1\log r+k_2$,
we have $c$ is of the form
\[
c(r) =
\begin{cases}
1 & \text{for $r \le f$,} \\
-\frac{K}{4}r^2+k_1\log r+k_2 & \text{for $f \le r \le p$,} \\
\frac{K}{4}r^2+k_3\log r+k_4 & \text{for $p \le r \le b$,} \\
0 & \text{for $r \ge b$} \\
\end{cases}
\]
with some constants $k_1, k_2, k_3, k_4$.
Since $c(r)$ and its derivative $c_r(r)$ are continuous,
we have the boundary conditions
\[
-\frac{K}{4}f^2+k_1\log f+k_2 = 1,
\quad -\frac{K}{2}f+\frac{k_1}{f} = 0,
\]
\[
\frac{K}{4}b^2+k_3\log b+k_4 = 0,
\quad \frac{K}{2}b+\frac{k_3}{b} = 0,
\]
\[
-\frac{K}{4}p^2+k_1\log p+k_2 = \frac{K}{4}p^2+k_3\log p+k_4,
\quad -\frac{K}{2}p+\frac{k_1}{p} = \frac{K}{2}p+\frac{k_3}{p}.
\]
Solving them, we obtain
\[
k_1 = \frac{K}{2}f^2,
\quad k_2 = 1+\frac{K}{4}f^2-\frac{K}{2}f^2\log f,
\]
\[
k_3 = -\frac{K}{2}b^2,
\quad k_4 = -\frac{K}{4}b^2+\frac{K}{2}b^2\log b.
\]
\[
p^2 = \frac{f^2+b^2}{2},
\quad K = \frac{2}{f^2\log f+b^2\log b-2 p^2\log p}.
\]
Therefore, we can construct a function $c$ with the constant $K$ as above
and obtain the estimate
\[
\int_\Omega \abs*{\nabla \vb*{d}}^2 \le \frac{1}{2}\pi(b^2-f^2) K\norm{\vb*{d}}_{L^\infty(D)}^2.
\]
Now, let us estimate $K$ by a simple expression.
Since the function $f(x) = x \log x$ is convex for $x > 0$ and moreover it satisfies
\[
f''(x) = \frac{1}{x} > 0,
\quad f''''(x) = \frac{2}{x^3} > 0,
\]
we have the inequality of convexity
\[
\frac{f(x)+f(y)}{2}-f\left(\frac{x+y}{2}\right)
> \frac{1}{8}f''\left(\frac{x+y}{2}\right)(y-x)^2.
\]
Thus, we obtain
\[
K \le \frac{8(b^2+f^2)}{(b^2-f^2)^2}
\]
and the desired estimate of this lemma.
\end{proof}

By using this lemma, it follows that
\[
\int_\Omega \abs*{\operatorname{div} (\vb*{s}-\bar{\vb*{s}})}^2
\le 2\int_\Omega \abs*{\nabla (\vb*{s}-\bar{\vb*{s}})}^2
\le 2\frac{4\pi (b_r^2+f_r^2)}{b_r^2-f_r^2} M_a^2.
\]
To summarize,
\[
\int_\Omega \abs{\frac{1}{T_{D, \Omega}^a}-\frac{1}{T_{\bar{D}, \Omega}^a}}^2
\le \frac{2\pi (b_r^2+f_r^2)}{b_r^2-f_r^2} a M_a^2.
\]

Now, let us estimate $\sqrt{a}M_a$.
We prepare the following lemma on the ratio of modified Bessel functions, which is a slight modification of the argument in \cite{r_ln10}.

\begin{lemma}
We have
\[
(\sqrt{x}\exp(x)K_1(x))' \le 0
\]
for all $x > 0$.
In particular, the inequality
\[
\frac{K_1(y)}{K_1(x)} \le \sqrt{\frac{x}{y}}\exp(-(y-x))
\]
holds for all $0 < x \le y$.
\end{lemma}

\begin{proof}
Direct calculation shows
\[
\begin{aligned}
(\sqrt{x}\exp(x)K_1(x))'
&= \frac{1}{2\sqrt{x}}\exp(x)[(2 x+1)K_1(x)+2 x K_1'(x)] \\
&= \frac{1}{2\sqrt{x}}\exp(x)[(2 x-1)K_1(x)-2 x K_0(x)].
\end{aligned}
\]
It follows from the inequality \eqref{e_besselk} that
\[
(\sqrt{x}\exp(x)K_1(x))'
\le \frac{1}{2\sqrt{x}}\exp(x)\left[(2 x-1)-2\left(x-\frac{1}{2}\right)\right] K_1(x) \\
= 0.
\]
We obtain the desired result.
\end{proof}

In view of this inequality,
we have
\[
\begin{aligned}
\sqrt{a} M_a
&= \frac{k(f_r+f_l) I_0 K_1\left(\frac{b_r}{\sqrt{a}}\right)}{k (f_r+f_l) I_0 K_0+2 (f_l I_1 K_0+f_r I_0 K_1)} \\
&\le \frac{k(f_r+f_l) I_0 K_1}{k (f_r+f_l) I_0 K_0+2 (f_l I_1 K_0+f_r I_0 K_1)} \sqrt{\frac{f_r}{b_r}}\exp\left(-\frac{b_r-f_r}{\sqrt{a}}\right).
\end{aligned}
\]
We claim that
\[
C_a
:= \frac{k(f_r+f_l) I_0 K_1}{k (f_r+f_l) I_0 K_0+2 (f_l I_1 K_0+f_r I_0 K_1)}
\le 1.
\]
Indeed, dropping the term of $I_1 K_0 \ge 0$ and recalling $k = \frac{f_r-f_l}{\sqrt{a}} > 0$, we have
\[
C_a
\le \frac{k(f_r+f_l) I_0 K_1}{k (f_r+f_l) I_0 K_0+2 f_r I_0 K_1}
= \frac{1}{\frac{K_0\left(\frac{f_r}{\sqrt{a}}\right)}{K_1\left(\frac{f_r}{\sqrt{a}}\right)}+\frac{2 f_r}{f_r^2-f_l^2}\sqrt{a}}.
\]
Then, the inequality \eqref{e_besselk} shows
\[
C_a
\le \frac{1}{1-\frac{\sqrt{a}}{2 f_r}+\frac{2 f_r}{f_r^2-f_l^2}\sqrt{a}}
= \frac{1}{1+\frac{3 f_r^2+f_l^2}{2 f_r(f_r^2-f_l^2)}\sqrt{a}}
\le 1.
\]

Therefore, we have
\[
\sqrt{a} M_a
\le \sqrt{\frac{f_r}{b_r}}\exp\left(-\frac{b_r-f_r}{\sqrt{a}}\right)
\]
and it finally follows that:
\[
\begin{aligned}
\norm{\frac{1}{T_{D, \Omega}^a}-\frac{1}{\bar{T}_{\Omega}}}_{L^2(\Omega)}
&\le \norm{\frac{1}{T_{D, \Omega}^a}-\frac{1}{T_{\mathbb{R}^2, \Omega}^a}}_{L^2(\Omega)}+\norm{\frac{1}{T_{\mathbb{R}^2, \Omega}^a}-\frac{1}{\bar{T}_{\Omega}}}_{L^2(\Omega)} \\
&\le \frac{1}{T^2}\left(\int_\Omega \abs{T_{\mathbb{R}^2, \Omega}^a-\bar{T}_{\Omega}}^2\right)^{\frac{1}{2}}+\left(\int_\Omega \abs{\frac{1}{T_{D, \Omega}^a}-\frac{1}{T_{\mathbb{R}^2, \Omega}^a}}^2\right)^{\frac{1}{2}} \\
&\le 2\frac{f_r}{f_l}\frac{\sqrt{\abs{\Omega}}}{T^2} \sqrt{a}+\sqrt{\frac{2\pi(b_r^2+f_r^2)}{b_r^2-f_r^2}}\sqrt{\frac{f_r}{b_r}}\exp\left(-\frac{b_r-f_r}{\sqrt{a}}\right).
\end{aligned}
\]
Since
\[
\sqrt{\frac{b_r^2+f_r^2}{b_r+f_r}}
\le \sqrt{b_r+f_r}
\le \sqrt{2 b_r},
\]
we can conclude that the estimate in Theorem \ref{t_anu_gen} holds.

\section*{Acknowledgment}
The first author is partially supported by JSPS KAKENHI Grant Number 25K17275.
The second author is partially supported by JSPS KAKENHI Grant Number JP23H03800.

\end{document}